\documentclass[11pt, amsfonts]{amsart}

\usepackage{amsmath,amssymb,amsthm,enumerate,comment}
\usepackage[all]{xy}
\usepackage{color}
\usepackage{graphicx}
\usepackage[T1]{fontenc}

\SelectTips{eu}{12}

\theoremstyle{definition}
\newtheorem{theorem}{Theorem}[section]
\newtheorem{definition}[theorem]{\rm Definition}
\newtheorem{lemma}[theorem]{Lemma}
\newtheorem{proposition}[theorem]{Proposition}
\newtheorem{corollary}[theorem]{Corollary}
\newtheorem{remark}[theorem]{\rm Remark}

\makeatletter
\@addtoreset{equation}{section} 
\makeatother

\DeclareMathOperator{\Diff}{Diff}

\DeclareMathOperator{\id}{id}

\DeclareMathOperator{\Symp}{Symp}
\DeclareMathOperator{\rel}{rel}

\DeclareMathOperator{\Hom}{Hom}
\DeclareMathOperator{\Ham}{Ham}

\DeclareMathOperator{\grp}{grp}
\DeclareMathOperator{\cell}{cell}

\DeclareMathOperator{\FS}{FS}

\newcommand{\RR}{\mathbb{R}}
\newcommand{\ZZ}{\mathbb{Z}}

\newcommand{\GG}{\Gamma}
\renewcommand{\gg}{\gamma}
\newcommand{\tSymp}{\mathrm{S}\widetilde{\mathrm{ymp}}}
\newcommand{\tHam}{\mathrm{H}\widetilde{\mathrm{am}}}

\title[Dixmier-Douady class, action homomorphism, and group cocycles]{The Dixmier-Douady class, the action homomorphism, and group cocycles on the symplectomorphism group}
\author{Shuhei Maruyama}
\address{Graduate School of Mathematics,
Nagoya University, Japan}
\email{m17037h@math.nagoya-u.ac.jp}

\begin{document}

\begin{abstract}
  Let $X$ be a one-connected and integral symplectic manifold.
  In this paper, we construct and study a two-cocycle and three-cocycle on the symplectomorphism group of $X$.
  In particular, by using these cocycles, we clarify the relationship between Weinstein's action homomorphism and the universal Dixmier-Douady class of flat symplectic fibrations.
\end{abstract}

\maketitle

\section{Introduction}

In \cite{ismagilov_losik_michor06}, Ismagilov, Losik, and Michor constructed an $\RR$-valued group two-cocycle on symplectomorphism groups.
Their cocycle is defined as follows.
Let $(X, \omega = d\eta)$ be a connected and exact symplectic manifold whose first Betti number is zero.
For a base point $x_0 \in X$, the cocycle $c_{x_0, \eta}$ is given by
\[
  c_{x_0, \eta}(g, h) = \int_{x_0}^{h(x_0)} g^*\eta - \eta,
\]
where $\int_{x_0}^{h(x_0)}$ denotes the integration along a path from $x_0$ to $h(x_0)$.
This cocycle is well defined by the assumption on the Betti number.
Moreover, the cohomology class represented by the cocycle is independent of the choice of $x_0$ and $\eta$ (see \cite{ismagilov_losik_michor06}).
Cocycles of this type have been studied in several papers (\cite{MR2945578}, \cite{MR2854098}, \cite{MR4301317}).

In this paper, we construct and study an $S^1$-valued group two-cocycle $a_{x_0, \alpha}$ and a $\ZZ$-valued group three-cocycle $c_{x_0, \sigma, w}$, which are analogous to Ismagilov, Losik, and Michor's two-cocycle $c_{x_0, \eta}$.
These cocycles $a_{x_0, \alpha}$ and $c_{x_0, \sigma, w}$ are defined on the symplectomorphism group $\Symp(X,\omega)$ for a one-connected and integral symplectic manifold $(X, \omega)$.
Here the symplectic manifold is called \textit{integral} if the cohomology class of the symplectic form $\omega$ comes from an integral cohomology class of $X$.

The cocycle $a_{x_0, \alpha}$ is defined as follows.
Since the symplectic manifold is integral, the cohomology class represented by the symplectic form is trivial in the second cohomology group with coefficients in $S^1 = \RR/\ZZ$.
Hence there exists a singular one-cochain $\alpha \in C^1(X;S^1)$ whose coboundary $\delta \alpha$ is equal to the ``symplectic form $\omega$ modulo integer'' (see Subsection \ref{subsec:2cocycle} for detail).
Thus, we can construct an $S^1$-valued cocycle which is analogous to Ismagilov, Losik, and Michor's cocycle.
In fact, for $g, h \in \Symp(X,\omega)$, we set
\begin{align*}
  a_{x_0,\alpha}(g,h) = \int_{x_0}^{h(x_0)} g^* \alpha - \alpha,
\end{align*}
where $x_0 \in X$ is a base point.

Let $w$ be a singular two-cocycle with coefficients in $\ZZ$ which represents $[\omega]$.
Then, the group three-cocycle $c_{x_0, \sigma, w}$ is defined as the following form:
\begin{align*}
  c_{x_0, \sigma, w}(f,g,h) = \int_{\Delta_{x_0, \sigma}(g,h)} f^*w - w,
\end{align*}
(see Subsection \ref{subsection:three_cocycle} for the definition of $\Delta_{x_0, \sigma}$).
Note that the cocycle $c_{x_0, \sigma, w}$ takes values in $\ZZ$ since $w$ is a singular cocycle with coefficients in $\ZZ$.

As Ismagilov, Losik, and Michor's two-cocycle, the cohomology classes of $a_{x_0, \alpha}$ and $c_{x_0, \sigma, w}$ are independent of the choice of $x_0$, $\alpha$, $\sigma$, and $w$. 
Therefore, the group cohomology classes $[a_{x_0, \alpha}] \in H_{\grp}^2(\Symp(X,\omega);S^1)$ and $[c_{x_0, \sigma, w}] \in H_{\grp}^3(\Symp(X,\omega);\mathbb{Z})$ depend only on the symplectic structure on $X$.

For a group $G$, the cohomology group $H_{\grp}^*(G;\ZZ)$ is isomorphic to the singular cohomology group $H^*(BG^{\delta};\ZZ)$ of the classifying space of the discrete group $G^{\delta}$ (see \cite{brown82}).
Moreover, a cohomology class of the classifying space $BG^{\delta}$ can be seen as a (universal) characteristic class of flat (or foliated) $G$-bundles. 
Thus, the cohomology class $[c_{x_0, \sigma, w}] \in H_{\grp}^3(\Symp(X,\omega);\mathbb{Z})$ is considered as a characteristic class of flat $\Symp(X, \omega)$-bundles (such fiber bundles are called \textit{flat symplectic fibrations}).
From this point of view, we show the following:

\begin{theorem}\label{main_theorem}
  Let $D \in H^3(B\Symp(X,\omega);\ZZ)$ be the universal Dixmier-Douady class and $\iota^* \colon H^3(B\Symp(X,\omega);\ZZ) \to H_{\grp}^3(\Symp(X,\omega);\ZZ)$ the canonical map.
  Then, the group cohomology class $[c_{x_0, \sigma, w}] \in H_{\grp}^3(\Symp(X,\omega);\mathbb{Z})$ is equal to the negative of the universal Dixmier-Douady class $D^{\delta} = \iota^* D$ of flat symplectic fibrations.
\end{theorem}

Note that the universal Dixmier-Douady class $D \in H^3(B\Symp(X,\omega);\ZZ)$ is a characteristic class defined as an obstruction to the existence of a prequantum structure on fibrations (see Section \ref{sec:DD_class}, Definition \ref{def:DD_class}).

Let $\Ham_c(X,\omega)$ be the compactly supported Hamiltonian diffeomorphism group.
In \cite{MR990190}, Weinstein defined a homomorphism $\mathbf{A}$ on the fundamental group of $\Ham_c(X,\omega)$, which is called the \textit{action homomorphism}.
In our setting, the target of the action homomorphism is the circle $S^1 = \RR/\ZZ$.
Note that the first cohomology group is isomorphic to the space of homomorphism to the coefficients.
Thus the action homomorphism $\mathbf{A}$ is considered as an element of $H_{\grp}^1(\pi_1(\Ham_c(X,\omega));S^1)$.

A derivation in the Hochschild-Serre spectral sequence for the group extension
\[
  0 \to \pi_1(\Ham_c(X,\omega)) \to \tHam_c(X,\omega) \to \Ham_c(X,\omega) \to 1
\]
defines a homomorphism
\[
  d_2^{0,1} \colon H_{\grp}^1(\pi_1(\Ham_c(X,\omega));S^1) \to H_{\grp}^2(\Ham_c(X,\omega);S^1).
\]
By this map, we obtain a second group cohomology class $d_{2}^{0,1}(\mathbf{A}) \in H_{\grp}^2(\Ham_c(X,\omega);S^1)$.
For the inclusion $i \colon \Ham_c(X,\omega) \to \Symp(X,\omega)$, let
\[
  i^* \colon H_{\grp}^2(\Symp(X,\omega);S^1) \to H_{\grp}^2(\Ham_c(X,\omega);S^1)
\]
be the pullback.
Then the following holds.

\begin{theorem}\label{thm:action_hom_class_a}
  The cohomology class $d_2^{0,1}(\mathbf{A})$ is equal to the negative of the class $i^*[a_{x_0, \alpha}]$.
\end{theorem}

The short exact sequence $0 \to \ZZ \to \RR \to S^1 \to 0$ induces the cohomology long exact sequence
\[
  \cdots \to H_{\grp}^2(\Ham_c(X,\omega);\RR) \to H_{\grp}^2(\Ham_c(X,\omega);S^1) \xrightarrow{\mathbf{d}} H_{\grp}^3(\Ham_c(X,\omega);\ZZ) \to \cdots,
\]
where $\mathbf{d}$ is the connecting homomorphism.
Because the class $\mathbf{d}([a_{x_0, \alpha}])$ is equal to $-[c_{x_0, \sigma, w}]$ (see Theorem \ref{thm:third_cocycle}), we obtain the following:
\begin{corollary}\label{cor:action_hom_DD}
  The action homomorphism $\mathbf{A}$ corresponds to the negative of the universal Dixmier-Douady class $D^{\delta}$ on $\Ham_c(X,\omega)$ under the map
  \[
    \mathbf{d} \circ d_2^{0,1} \colon H_{\grp}^1(\pi_1(\Ham_c(X, \omega));S^1) \to H_{\grp}^3(\Ham_c(X,\omega);\ZZ).
  \]
\end{corollary}

In other words, Corollary \ref{cor:action_hom_DD} states the following;
the action homomorphism $\mathbf{A}$, 
the Dixmier-Douady classes $D$ and $D^{\delta}$, and the cocycles $a_{x_0, \alpha}$ and $c_{x_0, \sigma, w}$ of Ismagilov, Losik, and Michor's type are correspond each others in the following diagram:
\[
\xymatrix{
  H^3(B\Ham_c(X,\omega);\ZZ) \ar[r]^-{\iota^*} & H_{\grp}^3(\Ham_c(X,\omega);\ZZ) \\
  H_{\grp}^1(\pi_1(\Ham_c(X,\omega));S^1) \ar[r]^-{d_2^{0,1}} & H_{\grp}^2(\Ham_c(X,\omega);S^1) \ar[u]_-{\mathbf{d}}.
}
\]

Note that the corollary is analogous to Tsuboi's theorem in \cite{tsuboi00}, which states that the Calabi homomorphism on the symplectomorphism group of the disk corresponds to the Euler class on the orientation preserving diffeomorphism group $\Diff_+(S^1)$ under the map
\[
  d_2^{0,1} \colon H_{\grp}^1(\Symp(D,\omega)_{\rel};\RR) \to H_{\grp}^2(\Diff_+(S^1);\RR).
\]
Here the map $d_2^{0,1}$ is the derivation map in the Hochschild-Serre spectral sequence of the group extension $1 \to \Symp(D,\omega)_{\rel} \to \Symp(D,\omega) \to \Diff_+(S^1) \to 1$.

The organization of the present paper goes as follows:
In Section \ref{sec:Preliminaries}, we review group cohomology and the Hochschild-Serre spectral sequence.
In Section \ref{sec:cocycles}, the precise definitions of the cocycles are given.
In Section \ref{sec:action_hom}, we extend the action homomorphism to a map from the universal covering group.
Moreover, we prove Theorem \ref{thm:action_hom_class_a}.
In Section \ref{sec:prequantum_extension}, we explain a relation between the cocycle $a_{x_0, \alpha}$ and the prequantum extension of $\Symp(X,\omega)$.
In Section \ref{sec:DD_class}, we show Theorem \ref{main_theorem}.
In Section \ref{sec:examples}, we give an example where the classes treated in this paper are non-zero.

\section{Preliminaries}\label{sec:Preliminaries}
\subsection{Group cohomology}\label{subsec:group_coh}
Let $G$ be a group and $M$ a trivial $G$-module.
Then the set of all functions $C_{\grp}^p(G;M) = \{ c \colon G^p \to M \}$ from $p$-fold product $G^p$ to $M$ is called the {\it $p$-cochain group of $G$}.
The coboundary map $\delta \colon C_{\grp}^p(G;M) \to C_{\grp}^{p+1}(G;M)$ is defined by
\begin{align*}
  \delta c (g_1, \dots, g_{p+1}) =& c(g_2, \dots, g_{p+1}) + \sum_{i = 1}^{p}(-1)^i c(g_1, \dots, g_i g_{i+1}, \dots, g_{p+1}) \\
  & +(-1)^{p+1}c(g_1, \dots, g_p)
\end{align*}
for $p > 0$ and $\delta = 0$ for $p = 0$.
The cohomology of the cochain complex $(C_{\grp}^*(G;M),\delta)$ is  denoted by $H_{\grp}^*(G;M)$ and called the {\it group cohomology of $G$}.

For a group $G$, the group cohomology $H_{\grp}^*(G;M)$ is isomorphic to the singular cohomology $H^*(BG^{\delta};M)$ of the classifying space $BG^{\delta}$, where $G^{\delta}$ denotes the group $G$ with discrete topology (see \cite{brown82}).

For a topological group $G$, the identity homomorphism $G^{\delta} \to G$ induces a continuous map $\iota \colon BG^{\delta} \to BG$.
Then, the following theorem is known, which we will use in the proof of Theorem \ref{thm:non-trivial_ch_class}.

\begin{theorem}[\cite{milnor83}]\label{theorem:milnor_made_discrete}
  Let $G$ be a finite dimensional Lie group with finite connected
  components.
  Then the pullback $\iota^* \colon H^*(BG;\mathbb{Z}) \to H^*(BG^{\delta};\mathbb{Z})$ is injective.
\end{theorem}

The second group cohomology of $G$ is closely related to the central extensions of $G$.
An exact sequence $1 \to M \xrightarrow{i} \Gamma \xrightarrow{\pi} G \to 1$ of groups is called a {\it central $M$-extension of} $G$ if the image $i(M)$ is contained in the center of $\Gamma$.
It is known that the second group cohomology $H_{\grp}^2(G;M)$ is bijective to the set of equivalence classes of central $M$-extensions of $G$;
\[
  H_{\grp}^2(G;M) \cong \{ \text{central $M$-extensions of $G$}\} / \{ \text{splitting extensions} \}
\]
(see \cite{brown82}).
For a central $M$-extension $\Gamma$, the corresponding cohomology class $e(\Gamma)$ is defined as follows.
Let $s \colon G \to \Gamma$ be a section of the projection $\pi \colon \Gamma \to G$.
For any $g, h \in G$, the value $s(g)s(h)s(gh)^{-1}$ is in $i(M) \cong M$.
By setting
\begin{align}\label{euler_cocycle}
  c(g,h) = s(g)s(h)s(gh)^{-1},
\end{align}
we obtain a two-cochain $c \in C_{\grp}^2(G;M)$ on $G$.
It is easily seen that the cochain $c$ is a cocycle, and its cohomology class $[c]$ does not depend on the choice of sections.
We set $e(\Gamma) = [c] \in H_{\grp}^2(G;M)$.
This cohomology class $e(\Gamma)$ is the class that corresponds to the central $M$-extension $\Gamma$.

\begin{lemma}\label{lemma:id_connecting_hom}
  Let $H_{\grp}^1(S^1;S^1) \xrightarrow{\delta}
  H_{\grp}^2(S^1;\mathbb{Z})$ be the connecting homomorphism and $e(\mathbb{R}) \in H_{\grp}^2(S^1;\mathbb{Z})$ be the cohomology class corresponding to the central $\mathbb{Z}$-extension $0 \to \mathbb{Z} \to \mathbb{R} \to S^1 \to 0$.
  Then, the class $\delta(\id_{S^1})$ is equal to $-e(\mathbb{R})$.
\end{lemma}
\begin{proof}
  Let $l\colon S^1 = \mathbb{R}/\mathbb{Z} \to [0,1) \subset \mathbb{R}$ be the (uniquely defined) section of the projection $\mathbb{R} \to S^1$.
  By the definition of the connecting homomorphism, we have $\delta(\id_{S^1}) = [\delta l] \in
  H_{\grp}^2(S^1;\mathbb{Z})$.
  By the definition of the cocycle (\ref{euler_cocycle}) of $e(\mathbb{R})$, the class $-[\delta l]$ is equal to $e(\mathbb{R})$.
\end{proof}

\subsection{Hochschild-Serre spectral sequence}

For a group extension, there exists a spectral sequence called the \textit{Hochschild-Serre spectral sequence}.

\begin{theorem}[\cite{hochschild_serre53}]\label{thm:HS_s.s}
  For a group extension $1 \to K \to \GG \xrightarrow{\pi} G \to 1$ and a trivial $\GG$-module $M$, there exists a first quadrant spectral sequence $(E_r^{p,q},d_r^{p,q})$ with $E_2^{p,q} \cong H^p(G;H^q(K;M))$ which converges to $H^{p+q}(\GG;M)$.
\end{theorem}

\begin{remark}
  The Hochschild-Serre spectral sequence of a group extension
  \begin{align*}
    1 \to K \to \GG \xrightarrow{\pi} G \to 1,
  \end{align*}
  is isomorphic to the Serre spectral sequence of the fibration
  \begin{align*}
    BK^{\delta} \to B\GG^{\delta} \to BG^{\delta}
  \end{align*}
  of the classifying spaces, where $K^{\delta}$, $\GG^{\delta}$, and $G^{\delta}$ denote the topological groups with discrete topology (see \cite{MR1634407} for example).
\end{remark}

We briefly recall the definition of the derivations in Hochschild-Serre spectral sequence (see \cite{hochschild_serre53} for details).
Let $A^n(\GG;M)$ be the set of all normalized cochains on $\GG$, that is, any cochain $f$ in $A^n(\GG;M)$ satisfies
\[
  f(\gg_1, \dots, \gg_n) = 0
\]
whenever $\gg_j = 1_{\GG}$ for some $1 \leq j \leq n$, where $1_{\GG}$ is the unit element of $\GG$.
We define subsets $A_j^{*n} \subset A^n(\GG;M)$ by
\begin{align*}
  A_j^{*n} = \{ f \in A^n(\GG;M) \mid f(\gg_1,& \dots, \gg_n) \text{ depends only on }\\
  &\gg_1, \dots, \gg_{n-j} \text{ and } \pi(\gg_{n-j+1}), \dots, \pi(\gg_n) \}
\end{align*}
for $0 \leq j \leq n$, $A_j^{*n} = A^n(\GG;M)$ for $j < 0$, and $A_j^{*n} = 0$ for $j > n$.
Then, subsets $A_j^{*n}$ define a decreasing filtration
\[
  A^n(\GG;M) = A_{0}^{*n} \supset A_1^{*n} \supset \dots \supset A_n^{*n} \supset A_{n+1}^{*n} = 0.
\]
We set
\begin{align}\label{eq:def_of_Zrpq}
  Z_r^{p,q} = \{ f \in A_{p}^{*p+q} \mid \delta f \in A_{p+r}^{*p+q+1} \}.
\end{align}
Then, $E_r^{p,q}$ 
is defined by
\begin{align}\label{eq:def_of_Erpq}
  E_r^{p, q} = Z_r^{p,q} / (Z_{r-1}^{p+1, q-1} + \delta A_{p-r+1}^{*p+q-1}).
\end{align}
The derivation map
\[
  d_r^{p,q} : E_r^{p,q} \to E_r^{p+r, q-r+1}
\]
is induced from the coboundary map
\[
  \delta : A^{p+q}(\GG;M) \to A^{p+q+1}(\GG;M).
\]
Let
\begin{align}\label{can_isom_E_2}
  \phi \colon Z_{2}^{p.q} \to C_{\grp}^p(G;H^q(K;M)),
\end{align}
be a map defined by
\begin{align*}
  \phi(f)(g_1, \cdots, g_p) = [f(\ast, \cdots, \ast, g_1, \cdots, g_p)] \in H^q(K;M),
\end{align*}
where we regard $f(\ast, \cdots, \ast, g_1, \cdots, g_p)$ as a group $q$-cocycle on $K$.
Then, this map $\phi$ induces the isomorphism $E_2^{p,q} \cong H^p(G;H^q(K;M))$ in Theorem \ref{thm:HS_s.s}.
The transgression map is described as follows.

For $r \geq 1$, the derivation map $d_r^{0,r-1} \colon E_r^{0, r-1} \to E_r^{r,0}$ is called the \textit{transgression map}.
\begin{lemma}\label{lem:transgression_description}
  Let $f \in A^r(\GG;M)$ be a normalized cochain.
  Assume that there exists a cocycle $c \in C_{\grp}^{r+1}(G;M)$ such that the pullback $\pi^* c$ is equal to $\delta f$.
  Then, the cochain $f$ defines an element $[f]_r$ of $E_r^{0, r-1}$ and the cohomology class $d_r^{0,r-1}([f]_r)$ is equal to $[\pi^* c]_r \in E_r^{r,0}$, where the bracket $[\, \cdot \, ]_r$ denotes the equivalence class in the quotient (\ref{eq:def_of_Erpq}).
  Moreover, if $E_r^{0,r-1} \cong E_2^{0,r-1} \cong H_{\grp}^{r-1}(K;M)$ and $E_r^{r,0} \cong E_2^{r,0} \cong H_{\grp}^r(G;M)$, then the cohomology classes of $H_{\grp}^{r-1}(K;M)$ and $H_{\grp}^r(G;M)$ corresponding to $[f]_r$ and $d_r^{0,r-1}([f]_r)$ are represented by $f|_{K^{r-1}}$ and $c$, respectively.
\end{lemma}

Since the proof is straightforward, we omit it.

\begin{lemma}\label{lemma:commutative_diagram}
  Let $0 \to S^1 \to \Gamma \xrightarrow{\pi} G \to 1$ be a central $S^1$-extension.
  Then, the diagram
  \[
  \xymatrix{
  H_{\grp}^1(S^1;S^1) \ar[r]^-{d_2^{0,1}} \ar[d]^{\delta} & H_{\grp}^2(G;S^1) \ar[d]^{\delta} \\
  H_{\grp}^2(S^1;\mathbb{Z}) \ar[r]^-{-d_3^{0,2}} & H_{\grp}^3(G;\mathbb{Z})
  }
  \]
  commutes, where each $\delta$ is the connecting homomorphism and $d_2^{0,1}$ and $d_3^{0,2}$ are the transgression maps of the Hochschild-Serre spectral sequence of the central $S^1$-extension.
\end{lemma}

\begin{proof}
  For an element $\varphi \in H_{\grp}^1(S^1;S^1)$ and for a section $s \colon G \to \Gamma$ satisfying $s(1_{G}) = 1_{\Gamma}$, we define a normalized cochain $\varphi_s \colon \Gamma \to S^1$ by setting
  \[
    \varphi_s(\gamma) = \varphi(\gamma \cdot (s\pi (\gamma))^{-1}).
  \]
  Then, the restriction $\varphi_s|_{S^1}$ is equal to $\varphi$, and the cochain $\varphi_s$ is contained in $Z_2^{0,1}$.
  Hence there exists a normalized cocycle $c \in A^2(G;S^1)$ satisfying
  \[
    \pi^* c = \delta \varphi_s,
  \]
  and the cohomology class $[c] \in H_{\grp}^2(G;S^1)$ is equal to $d_2^{0,1}(\varphi)$.
  By the definition of the connecting homomorphism $\delta \colon H_{\grp}^2(G;S^1) \to H_{\grp}^3(G;\ZZ)$, we have $\delta d_2^{0,1} (\varphi) = \delta [c] = [\delta (l \circ c)] \in H_{\grp}^3(G;\mathbb{Z})$,
  where $l\colon S^1 \to \RR$ is the section defined in the proof of Lemma \ref{lemma:id_connecting_hom}.

  By the definition of the connecting homomorphism $\delta \colon H_{\grp}^1(S^1;S^1) \to H_{\grp}^2(S^1;\ZZ)$, the class $\delta (\varphi)$ is equal to $[\delta (l \circ \varphi)] \in H_{\grp}^2(S^1;\ZZ)$.
  Let $c' \in A^2(\Gamma, \mathbb{Z})$ be a normalized cochain defined by
  \[
    c' = \delta (l \circ \varphi_s) - \pi^*(l \circ c).
  \]
  Then, the restriction $c'|_{S^1\times S^1}$ is equal to $\delta (l \circ \varphi)$, and the coboundary $\delta c'$ is equal to $-\delta(\pi^*(l \circ c) = -\pi^*(\delta (l \circ c))$.
  Hence, by Lemma \ref{lem:transgression_description}, we have
  \[
    d_3^{0,2} (\delta(\varphi)) = - [\delta(l \circ c)] \in H_{\grp}^3(G; \mathbb{Z}),
  \]
  and the lemma follows.
\end{proof}

In terms of the Hochschild-Serre spectral sequence, the cohomology
class $e(\Gamma)$ is described as follows.

\begin{lemma}\label{lemma:identity-Euler_class}
  Let $1 \to M \to \Gamma \xrightarrow{\pi} G \to 1$ be a central $M$-extension of $G$ and $E_r^{p,q}$ the Hochschild-Serre spectral sequence of the central extension.
  Let
  \[
    d_2^{0,1} \colon H_{\grp}^1(M;M) = E_2^{0,1} \to E_2^{2,0} = H_{\grp}^2(G;M)
  \]
  be the derivation of the spectral sequence.
  Then, the cohomology class $e(\Gamma) \in H_{\grp}^2(G;M)$ is equal to $-d_2^{0,1}(\id_M)$.
\end{lemma}

\begin{proof}
  Let $s \colon G \to \Gamma$ be a section satisfying $s(1_G) = 1_{\GG}$, where $1_G$ and $1_{\Gamma}$ are the unit elements of $G$ and $\Gamma$, respectively.
  We define a cochain $f \in C_{\grp}^1(\GG;M)$ by setting $f(\gamma) = \gamma \cdot(s\pi(\gamma))^{-1}$.
  Let $c \in C_{\grp}^2(G;M)$ be a cocycle defined by (\ref{euler_cocycle}).
  Note that the cocycles $f$ and $c$ are normalized since $s(1_G) = 1_{\GG}$.
  Since $M$ is in the center of $\GG$, we have
  \begin{align}\label{eq1}
    \delta f(\gg_1, \gg_2) &= f(\gg_2) - f(\gg_1 \gg_2) + f(\gg_1)\\
    &= -(\gg_1 \gg_2 \cdot (s\pi(\gg_1 \gg_2))^{-1} + s\pi(\gg_1) \cdot \gg_1^{-1} + s\pi(\gg_2) \cdot \gg_2^{-1}) \nonumber \\
    &= -s(\pi(\gg_1))s(\pi(\gg_2)) \gg_2^{-1} \gg_1^{-1} \cdot \gg_1 \gg_2 s(\pi(\gg_1)\pi(\gg_2))^{-1} \nonumber \\
    &= -s(\pi(\gg_1))s(\pi(\gg_2))s(\pi(\gg_1)\pi(\gg_2))^{-1} = -\pi^*c (\gg_1, \gg_2) \nonumber
  \end{align}
  for any $\gg_1, \gg_2 \in \GG$, where $c$ is the cocycle defined in (\ref{euler_cocycle}).
  Since $f|_{M} = \id_M$, we have $d_2^{0,1}(\id_M) = [-c] = -e(\GG)$ by Lemma \ref{lem:transgression_description}.
\end{proof}

The following lemma is well known.
\begin{lemma}\label{lemma:obstruction-euler_class}
  Let $1 \to A \to \Gamma \to G \to 1$ be a central $A$-extension of $G$ and
  \begin{align}\label{fib_classifying_space_d}
    BA^{\delta} \to B\Gamma^{\delta} \to BG^{\delta}
  \end{align}
  the corresponding fibration of classifying spaces of discrete groups.
  Then, the primary obstruction class $\mathfrak{o}(\id, B\Gamma^{\delta}) \in H^2(BG^{\delta};A)$ coincides with the class $e(\Gamma) \in H_{\grp}^2(G;A)$ under the canonical isomorphism $H^2(BG^{\delta};A) \cong H_{\grp}^2(G;A)$.
\end{lemma}

\section{Cocycles}\label{sec:cocycles}

\subsection{The group two-cocycle}\label{subsec:2cocycle}


Let $(X, \omega)$ be a one-connected and integral symplectic manifold.
Let $(\Omega^*(M),d)$ denote the de Rham complex and $(C^*(X;\mathbb{R}),\delta)$ the ($C^{\infty}$-)singular cochain complex with coefficients in $\mathbb{R}$.
Then, the canonical cochain map
\begin{align}\label{map_I}
  I\colon \Omega^n(M) \to C^n(X;\mathbb{R});\  \eta \to I_{\eta}
\end{align}
is defined by $I_{\eta}(\sigma) = \int_{\sigma} \eta$, where $\sigma$ is a $C^{\infty}$-singular $n$-simplex, and de Rham's theorem asserts that the map $I$ induces the isomorphism of cohomology.
The following property is standard.
\begin{lemma}\label{lemma:I_pullback_compati}
  The cochain map $I$ is compatible with pullback.
  In particular, all elements in $\Symp(X,\omega)$ preserve the singular two-cocycle $I_{\omega}$.
\end{lemma}

Let us consider the cohomology long exact sequence
\[
  \cdots \longrightarrow H^2(X;\mathbb{Z}) \longrightarrow H^2(X;\mathbb{R}) \overset{j_*}{\longrightarrow} H^2(X;S^1) \longrightarrow H^3(X;\ZZ) \longrightarrow \cdots
\]
induced from $0 \to \ZZ \to \RR \xrightarrow{j} S^1 \to 0$.
Since the symplectic form $\omega$ is integral, the cohomology class $j_*[I_{\omega}] = [jI_{\omega}]$ is equal to zero.
We take a singular one-cochain $\alpha \in C^1(X;S^1)$ such that $\delta \alpha = jI_{\omega}$.
By Lemma \ref{lemma:I_pullback_compati}, the cochain $g^*\alpha - \alpha$ is a cocycle for any $g \in \Symp(X,\omega)$.
For $g, h \in \Symp(X,\omega)$, we set
\[
  a_{x_0, \alpha}(g,h) = \int_{x_0}^{h(x_0)} g^*\alpha - \alpha 
\]
where $x_0$ is a point in $X$.
Here the symbol $\int_{x_0}^{h(x_0)}g^*\alpha - \alpha$ denotes the pairing of the cocycle $g^*\alpha - \alpha$ and a path from $x_0$ to $h(x_0)$.

\begin{proposition}
  The cochain $a_{x_0, \alpha} \in C_{\grp}^2(\Symp(X,\omega);S^1)$ is a cocycle.
  Moreover, the cohomology class $[a_{x_0, \alpha}]$ does not depend on the choice of $x_0$ and $\alpha$.
\end{proposition}

The proof is straightforward, cf. \cite[Theorem 3.1]{ismagilov_losik_michor06}.

\subsection{The group three-cocycle}\label{subsection:three_cocycle}
Let $(X, \omega)$ be a one-connected and integral symplectic manifold.
By the definition of the integral symplectic manifold, there exists a cohomology class $[\omega]_{\ZZ} \in H^2(X;\mathbb{Z})$ which corresponds to the class $[I_{\omega}] \in H^2(X;\RR)$ under the change of coefficients homomorphism.
Here $I$ is the map (\ref{map_I}).
\begin{remark}\label{rem:inj_of_ZtoR}
  The change of coefficients homomorphism $H^2(X;\ZZ) \to H^2(X;\RR)$ is injective.
  Indeed, the map is a part of the cohomology long exact sequence
  \[
    \cdots \to H^1(X;S^1) \to H^2(X;\ZZ) \to H^2(X;\RR) \to H^2(X;S^1) \to \cdots,
  \]
  where we identify the circle $S^1$ with the quotient $\RR/\ZZ$.
  Since $X$ is one-connected, the cohomology group $H^1(X;S^1)$ is trivial.
  Hence, the injectivity follows.
  In particular, the class $[\omega]_{\ZZ}$ is uniquely determined by the one-connectedness of $X$.
\end{remark}

For a base point $x_0 \in X$, let $\mathcal{P}_{x_0} = \{ \gamma \colon [0,1] \to X \mid \gamma(0) = x_0 \}$ denote the based path space of $X$.
Let $\sigma \colon X \to \mathcal{P}_{x_0}$ be a section of the projection $\mathcal{P}_{x_0} \to X ; \, \gamma \to \gamma(1)$.
Since $X$ is one-connectedness, there exists a disk $\Delta_{x_0, \sigma}(g,h)$ whose boundary is the one-chain $\sigma(g(x_0)) - \sigma(gh(x_0)) + g\sigma(h(x_0))$ for any $g,h \in \Symp(X,\omega)$.
Let $w \in C^2(X;\mathbb{Z})$ be a singular two-cocycle representing $[\omega]_{\mathbb{Z}} \in H^2(X;\mathbb{Z})$.
Then, we define a group cochain $c_{x_0, \sigma, w} \in C_{\grp}^3(\Symp(X,\omega);\mathbb{Z})$ by setting
\begin{align}\label{explicit_three_cocycle}
  c_{x_0, \sigma, w} (f,g,h) = \int_{\Delta_{x_0, \sigma}(g,h)} f^*w - w
\end{align}
for any $f,g,h \in \Symp(X,\omega)$.
Since the singular two-cocycle $f^*w - w$ is a coboundary, the value $c_{x_0, \sigma, w} (f,g,h)$ dose not depend on the choice of $\Delta_{x_0, \sigma}(g,h)$.

\begin{proposition}
  The group cochain $c_{x_0, \sigma, w}$ is a cocycle.
\end{proposition}

Since this is easily verified by a explicit calculation, we omit the proof.

\begin{proposition}\label{prop:indep_of_choices}
  The cohomology class $[c_{x_0, \sigma, w}] \in H_{\grp}^3(\Symp(X,\omega);\mathbb{Z})$ is independent of the choice of $x_0, \sigma$, and $w$.
\end{proposition}

\begin{proof}
  Let $x_1 \in X$ be another point, $\gamma$ a path from $x_0$ to $x_1$, and $\sigma' \colon X \to \mathcal{P}_{x_1}$ a section.
  Let $\Delta_{x_0, \sigma}'(g,h)$ be a disk whose boundary is $\sigma'(g(x_0)) - \sigma'(gh(x_0)) + g\sigma'(h(x_0))$ and $S(g)$ a disk whose boundary is $\gamma + \sigma'(g(x_1)) - g\gamma - \sigma(g(x_0))$.
  We set
  \[
    b(g,h) = \int_{S(h)} g^*w - w,
  \]
  then we have
  \[
    \int_{\Delta_{x_0, \sigma}(g,h)}f^*w - w - \int_{\Delta_{x_0, \sigma}'(g,h)} f^*w - w = \delta b(f,g,h).
  \]
  This implies that the cohomology class is independent of $x_0$ and $\sigma$.

  Let $w' \in C^2(X;\mathbb{Z})$ be another cocycle of $[w]_{\mathbb{Z}}$ and $v \in C^1(X;\mathbb{Z})$ a singular one-cochain satisfying $\delta v = w' - w$.
  We set
  \[
    b'(g,h) = \int_{\sigma(h(x_0))} g^*v - v,
  \]
  then we have
  \[
    \int_{\Delta_{x_0, \sigma}(g,h)}f^*w - w - \int_{\Delta_{x_0, \sigma}(g,h)}f^*w' - w' = \delta b'(f,g,h).
  \]
  This implies that the class is independent of the choice of $w$.
\end{proof}

\subsection{A relation between $c_{x_0, \sigma, w}$ and $a_{x_0, \alpha}$}
The short exact sequence $0 \to \mathbb{Z} \to \mathbb{R} \to S^1 \to 0$ induces the cohomology long exact sequence
\[
  \cdots \to H_{\grp}^2(\Symp(X,\omega);\mathbb{R}) \to H_{\grp}^2(\Symp(X,\omega);S^1) \xrightarrow{\delta} H_{\grp}^3(\Symp(X,\omega);\mathbb{Z}) \to \cdots,
\]
where $\delta$ is the connecting homomorphism.
Then, the following holds:
\begin{theorem}\label{thm:third_cocycle}
  The cohomology class $[c_{x_0, \sigma, w}]$ is equal to $-\delta [a_{x_0, \alpha}]$.
\end{theorem}

\begin{proof}
  Let $\overline{\alpha} \in C^1(X;\mathbb{R})$ be a lift of $\alpha \in C^1(X;S^1)$, that is, the cochain $\overline{\alpha}$ satisfies $j \overline{\alpha} = \alpha$ under the change of coefficients homomorphism $j \colon C^1(X;\mathbb{R}) \to C^1(X;S^1)$.
  We set
  \[
    \overline{a}_{x_0, \alpha}(g,h) = \int_{\sigma(h(x_0))} g^*\overline{\alpha} - \overline{\alpha}.
  \]
  Note that the group cochain $\overline{a}_{x_0, \alpha}$ in $C_{\grp}^2(\Symp(X,\omega);\mathbb{R})$ is a lift of the cocycle $a_{x_0, \alpha} \in C_{\grp}^2(\Symp(X,\omega);S^1)$.
  By the definition of the connecting homomorphism, one of the cocycles of $\delta [a_{x_0, \alpha}]$ is given by $\delta \overline{a}_{x_0, \alpha} \in C^3(\Symp(X,\omega);\mathbb{Z})$.
  For $f,g,h \in \Symp(X,\omega)$, we have
  \begin{align}\label{eq:a_c}
    &\delta \overline{a}_{x_0, \alpha}(f,g,h)\\ \nonumber
    = & \int_{\gamma(h(x_0))}g^*\overline{\alpha} - \overline{\alpha} - \int_{\gamma(h(x_0))}g^*f^*\overline{\alpha} - \overline{\alpha} + \int_{\gamma(gh(x_0))}f^*\overline{\alpha} - \overline{\alpha} - \int_{\gamma(g(x_0))}f^*\overline{\alpha} - \overline{\alpha}\\ \nonumber
    =& \int_{\gamma(gh(x_0))-\gamma(g(x_0))-g\gamma(h(x_0))} f^*\overline{\alpha} - \overline{\alpha} = \int_{\Delta_{x_0, \sigma}(g,h)}\delta (f^*\overline{\alpha} - \overline{\alpha})\\ \nonumber
    =& \int_{\Delta_{x_0, \sigma}(g,h)}f^* (\delta \overline{\alpha} -I_{\omega}) - (\delta \overline{\alpha} - I_{\omega}) = -  \int_{\Delta_{x_0, \sigma}(g,h)}f^*(I_{\omega}-\delta \overline{\alpha}) - (I_{\omega} - \delta \overline{\alpha}),
  \end{align}
  where the third equality follows from $f^*I_{\omega} = I_{\omega}$ (see Lemma \ref{lemma:I_pullback_compati}).
  Here the cocycle $I_{\omega} - \delta \overline{\alpha}$ is in $C^2(X;\mathbb{Z})$ and represents $[\omega]_{\mathbb{Z}}$.
  Indeed, by the change of coefficients homomorphism $j\colon C^*(X;\mathbb{R}) \to C^*(X;S^1)$, we have
  \[
    j(I_{\omega} - \delta \overline{\alpha}) = jI_{\omega} - \delta j \overline{\alpha} = jI_{\omega} - jI_{\omega} = 0.
  \]
  We set $w = I_{\omega} - \delta \overline{\alpha}$, then (\ref{eq:a_c}) implies that
  \[
    \delta \overline{a}_{x_0, \alpha} = - c_{x_0, \sigma, w},
  \]
  and the theorem follows.
\end{proof}

\begin{remark}
  Let $l \colon S^1 \to [0,1) \subset \mathbb{R}$ be the section of $\mathbb{R} \to S^1$, and we set $\overline{\alpha} = l \alpha$.
  For this lift $\overline{\alpha}$ of $\alpha$, the group two-cochain $\overline{a}$ gives rise to a bounded two-cochain.
  Thus, when $w = I_{\omega} - \delta \overline{\alpha} = I_{\omega} - l\alpha$, the group three-cocycle $c_{x_0, \sigma, w}$ is a bounded cocycle.
\end{remark}

\section{The action homomorphism and the group two-cocycle}\label{sec:action_hom}

The goal of this section is to prove Theorem \ref{thm:action_hom_class_a}, which states that the relation between the cohomology class $[a_{x_0, \alpha}]$ and Weinstein's action homomorphism.

Let us recall the definition of the action homomorphism $\mathbf{A} \colon \pi_1(\Ham_c(X, \omega)) \to S^1$.
For a loop $\{ \varphi_t \}_{0 \leq t \leq 1}$ in $\Ham(X, \omega)$ with $\varphi_0 = \varphi_1 = \id_{X}$, let $H_t$ be the time-dependent normalized Hamiltonian which generates the loop $\{ \varphi_t \}_{0 \leq t \leq 1}$.
For a base point $x_0 \in X$, let $\Delta_{x_0}$ denote a two-disk bounded by the loop $\varphi_t(x_0)$ in $X$.
Then, the map $\mathbf{A}$ is defined\footnote{Here we employ a sign convention used in \cite{MR3177909}.} by
\[
  \mathbf{A}([\varphi_t]) = \int_{\Delta_{x_0}} \omega - \int_{0}^{1} H_t(\varphi_t(x_0)) dt \  (\text{mod} \ \ZZ).
\]
This is a well-defined homomorphism (\cite{MR990190}; see also \cite[Section 2.1]{MR3177909}).

Let $(E_r^{p,q}, d_r^{p,q})$ be the Hochschild-Serre spectral sequence of the extension
\[
  0 \to \pi_1(\Ham(X, \omega)) \to \tHam(X,\omega) \to \Ham(X, \omega) \to 1.
\]
To see a cocycle of the class $d_2^{0,1}(\mathbf{A}) \in H_{\grp}^2(\Ham(X,\omega);S^1)$ by using Lemma \ref{lem:transgression_description}, we define a map
\[
  \mathcal{A}_{x_0, \alpha} \colon \tHam(X,\omega) \to S^1
\]
whose restriction to $\pi_1(\Symp(X, \omega)_0)$ is equal to $\mathbf{A}$.
Recall that the singular one-cochain $\alpha \in C^1(X;S^1)$ satisfies $\delta \alpha = jI_\omega$ (see Section \ref{subsec:2cocycle}).
For an element $h$ of $\tSymp(X,\omega)_0$, we take a path $\{ h_t \}_{0 \leq t \leq 1}$ in $\Symp(X,\omega)_0$ with $h_0 = \id_X$ which represents $h$.
Then, we set 
\[
  \mathcal{A}_{x_0, \alpha}(h) = \int_{\{ h_t(x_0) \}_{0 \leq t \leq 1}} \alpha - \left(\int_0^1 H_t(\varphi_t(x_0)) dt \ (\text{mod} \ \ZZ) \right).
\]
Here $H_t$ is the time-dependent normalized Hamiltonian which generates $h_t$ (see \cite{MR1826128} for example).
The well-defineness of the map $\mathcal{A}_{x_0}$ is shown in the same way as in $\mathbf{A}$.
Since $\delta \alpha = j I_{\omega}$, the restriction of $\mathcal{A}_{x_0, \alpha}$ to $\pi_1(\Symp(X,\omega)_0)$ is equal to the action homomorphism $\mathbf{A}$.

\begin{lemma}\label{lem:action_hom_cocycle}
  Let $p \colon \tSymp(X,\omega)_0 \to \Symp(X,\omega)_0$ be the universal covering.
  Then, the coboundary $\delta \mathcal{A}_{x_0, \alpha}$ is equal to $p^*i^*a_{x_0, \alpha}$.
\end{lemma}

\begin{proof}
  For elements of $f = [f_t]$ and $g = [g_t]$ of $\tSymp(X,\omega)_0$, the product $fg$ is represented by a path $h_t$ in $\Symp(X,\omega)_0$ defined by
  \[
    h_t =
    \begin{cases}
      f_{2t} & 0 \leq t \leq 1/2 \\
      f_1 g_{2t-1} & 1/2 \leq t \leq 1.
    \end{cases}
  \]
  Then we have
  \begin{align*}
    \delta \mathcal{A}_{x_0, \alpha} (f,g) = \int_{\{g_t(x_0)\}_{0 \leq t \leq 1}} \alpha - \int_{\{h_t(x_0)\}_{0\leq t \leq 1}} \alpha + \int_{\{f_t(x_0)\}_{0 \leq t \leq 1}} \alpha.
  \end{align*}
  Indeed, the second term in the definition of $\mathcal{A}_{x_0, \alpha}$ cancels each other out (see \cite[Proof of Theorem 1]{MR3177909}).
  By the definition of $h_t$, we obtain
  \begin{align*}
    \delta \mathcal{A}_{x_0, \alpha} (f, g) &= \int_{\{g_t(x_0)\}_{0 \leq t \leq 1}} \alpha - \int_{\{f_1(g_t(x))\}_{0 \leq t \leq 1}}\alpha \\
    &  = -\int_{x_0}^{g_1(x_0)}f_1^*\alpha - \alpha = -p^*i^* a_{x_0, \alpha}(f,g).
  \end{align*}
\end{proof}

\begin{proof}[Proof of Theorem \ref{thm:action_hom_class_a}]
  By Lemmas \ref{lem:transgression_description} and \ref{lem:action_hom_cocycle}, the cocycle $-i^*a_{x_0, \alpha}$ represents the class $d_2^{0,1}(\mathbf{A})$.
  Hence the theorem follows.
\end{proof}

\section{Prequantum extension}\label{sec:prequantum_extension}

For an integral symplectic manifold $(X, \omega)$, there exists a principal $S^1$-bundle $p \colon P \to X$ called the \textit{prequantization $S^1$-bundle} (see \cite{MR2397738} for example).
The total space $P$ of the prequantization $S^1$-bundle admits a connection form $\theta$ whose curvature form is equal to the symplectic form $\omega$.
Note that the first Chern class of the prequantization bundle is equal to the class $-[\omega]_{\ZZ}$.

The {\it quantomorphism group} $Q$ is defined by
\[
  Q = \{ \varphi \colon P \to P: \text{bundle automorphism, } \varphi^*\theta = \theta \}.
\]
If the symplectic manifold is one-connected, the quantomorphism group $Q$ defines a central $S^1$-extension of the symplectomorphism group (\cite[Theorem 2.2.2]{kostant70}):
\begin{align}\label{prequantum_extension}
  0 \to S^1 \to Q \to \Symp(X,\omega) \to 1.
\end{align}
Central extension (\ref{prequantum_extension}) is called the \textit{prequantum extension}.

Let $e(Q) \in H_{\grp}^2(\Symp(X, \omega);S^1)$ be the cohomology class corresponding to (\ref{prequantum_extension}). 
The following theorem clarifies a relation between the cohomology classes $e(Q)$ and $[a_{x_0, \alpha}]$.

\begin{theorem}\label{thm:cocycle_description}
  The cohomology class $[a_{x_0, \alpha}] \in H_{\grp}^2(\Symp(X,\omega);S^1)$ is equal to $-e(Q)$.
\end{theorem}

To prove Theorem \ref{thm:cocycle_description}, we prepare the following lemmas.
Let us recall that $p \colon P \to X$ is the prequantization $S^1$-bundle with the connection form $\theta \in \Omega^1(P)$ satisfying $d\theta = p^*\omega$.

\begin{lemma}\label{lemma:cohomologous_to_zero}
  A singular cochain $j I_{\theta} - p^* \alpha \in C^1(P;S^1)$ is a coboundary, that is, there exists a singular zero-cochain $\beta \in C^0(P;S^1)$ such that the equality $\delta \beta = j I_{\theta} - p^* \alpha$ holds.
\end{lemma}
\begin{proof}
  Since $X$ is one-connected, the inclusion $S^1 \to P$ as a fiber induces a surjection $\pi_1(S^1) \to \pi_1(P)$.
  Therefore, for any loop $\gamma'$ in $P$, there exists a loop $\gamma$ in a fiber of $P \to X$ which is homotopic to the loop $\gamma'$.
  Since $\alpha \in C^1(X;S^1)$ is a singular one-cochain satisfying $\delta \alpha = jI_{\omega}$, we have
  \begin{align}\label{eq:1}
    \int_{\gamma'} j I_{\theta} - p^* \alpha = \int_{\gamma} j I_{\theta} - p^* \alpha
    = \int_{\gamma} j I_{\theta} - \int_{p\gamma} \alpha = \int_{\gamma} j I_{\theta}.
  \end{align}
  where the symbol $\int_{\gamma}$ denotes the pairing of a cocycle and the cycle $\gamma$.
  Note that the last term of (\ref{eq:1}) is equal to the projection of the value $\int_{\gamma} \theta \in \mathbb{R}$ to $S^1$.
  Since the form $\theta$ is a connection form, the value $\int_{\gamma} \theta$ is in $\mathbb{Z}$, that is, $\int_{\gamma} j I_{\theta} = 0$ holds.
  Therefore, the cocycle $j I_{\theta} - p^* \alpha$ is cohomologous to zero, and the lemma follows.
\end{proof}

By Lemma \ref{lemma:cohomologous_to_zero}, there exists a singular cochain $\beta \in C^0(X;S^1)$ satisfying $\delta\beta = j I_{\theta} - p^* \alpha$.
For a base point $y_0 \in P$ with $p(y_0) = x_0$, we define a group cochain $\tau_{y_0} \in C_{\grp}^1(Q;S^1)$ by setting
\[
  \tau_{y_0}(\varphi) = \int_{y_0}^{\varphi (y_0)} j I_{\theta} - p^* \alpha = \beta(\varphi (y_0)) - \beta(y_0)
\]
for any $\varphi \in Q$.

\begin{lemma}\label{lemma:connection_cochain}
  The restriction $\tau|_{S^1} \colon S^1 \to S^1$ is equal to the identity map.
  Here we consider the group $S^1$ as a subgroup of the quantomorphism group $Q$ via the inclusion in (\ref{prequantum_extension}).
\end{lemma}

\begin{proof}
  For $u \in S^1 \subset Q$, let $\gamma$ be a path from $y_0$ to $y_0 \cdot u$ in the fiber over $p(y_0) = x_0$.
  Here the symbol $\cdot u$ denotes the right $S^1$-action equipped with the principal $S^1$-bundle.
  Then, we have
  \[
    \tau(u) = \int_{y_0}^{y_0 \cdot u} jI_{\theta} - p^*\alpha = j \int_{\gamma} \theta = u,
  \]
  and the lemma follows.
\end{proof}

\begin{lemma}\label{lemma:curvature}
  The equality
  \[
    \delta \tau = \pi^* (a_{x_0, \alpha}) \in C_{\grp}^2(Q;S^1)
  \]
  holds, where $\pi \colon Q \to \Symp(X,\omega)$ is the projection.
\end{lemma}

\begin{proof}
  For $\varphi, \psi \in Q$, we set $g = \pi(\varphi)$ and $h = \pi(\psi)$.
  By the definition of the quantomorphism group $Q$ and Lemma \ref{lemma:I_pullback_compati}, we have $\varphi^* jI_{\theta} = jI_{\theta}$ and $\psi^* jI_{\theta} = jI_{\theta}$.
  Then, we obtain
  \begin{align*}
    \delta \tau(\varphi, \psi) &= \int_{y_0}^{\psi (y_0)} -\int_{y_0}^{\varphi \psi (y_0)} + \int_{y_0}^{\varphi (y_0)} j I_{\theta} - p^* \alpha\\
    &= \int_{y_0}^{\psi (y_0)} - \int_{\varphi (y_0)}^{\varphi \psi (y_0)} j I_{\theta} - p^* \alpha\\
    &= \int_{y_0}^{\psi (y_0)} (j I_{\theta} - p^* \alpha) - \varphi^* (j I_{\theta} - p^* \alpha)\\
    &= \int_{y_0}^{\psi (y_0)} p^* g^* \alpha - p^* \alpha\\
    &= \int_{x_0}^{h (x_0)} g^*\alpha - \alpha = a_{x_0, \alpha} (g,h) = \pi^* a_{x_0, \alpha} (\varphi, \psi).
  \end{align*}
\end{proof}

\begin{proof}[Proof of Theorem \ref{thm:cocycle_description}]
  Let $E_r^{p,q}$ denote the Hochschild-Serre spectral sequence of
  \[
    0 \to S^1 \to Q \to \Symp(X,\omega) \to 1
  \]
  with coefficients in $S^1$.
  Then, the transgression map $d_2^{0,1} \colon E_2^{0,1} \to E_2^{2, 0}$ induces a map
  \[
    d_2^{0,1} \colon H_{\grp}^1(S^1;S^1) \cong E_2^{0,1} \to E_2^{2, 0} \cong H_{\grp}^2(\Symp(X,\omega);S^1).
  \]
  By Lemmas \ref{lem:transgression_description}, \ref{lemma:connection_cochain}, and \ref{lemma:curvature}, an equality $d_2^{0,1} (\id_{S^1}) = [a_{x_0, \alpha}]$ holds.
  On the other hand, by Lemma \ref{lemma:identity-Euler_class}, we have $d_2^{0,1} (\id_{S^1}) = -e(Q)$.
  Therefore the theorem follows.
\end{proof}

\section{The Dixmier-Douady class of symplectic fibrations and the group three-cocycle}\label{sec:DD_class}

Let $E \to B$ be a fiber bundle whose fiber is a symplectic manifold $(X, \omega)$.
The bundle $E \to B$ is called a \textit{symplectic fibration} if the structure group can be reduced to the symplectomorphism group $\Symp(X,\omega)$.
An extension of the structure group $\Symp(X,\omega)$ to $Q$ is called the \textit{prequantum structure} or \textit{prequantum lift} of $E$.
The exisitence of prequantum structures is detected by a characteristic class $D(E) \in H^3(B;\mathbb{Z})$ called the \textit{Dixmier-Douady class}.
The Dixmier-Douady class is defined by using the cohomology with coefficients in the sheaf of $S^1$-valued continuous functions (see \cite{brylinski_93} for the definition; see also \cite{savelyev_shelukhin20}).

Under the assumption that the fiber is one-connected, the Dixmier-Douady class can be defined by using the Serre spectral sequence.
Let $(X, \omega) \to E \to B$ be a symplectic fibration with connected base space $B$.
We assume that the fiber $X$ is a one-connected and integral symplectic manifold.
Let $(E_r^{p,q}, d_r^{p, q})$ be the cohomology Serre spectral sequence with coefficients in $\mathbb{Z}$.
Then the $E_2$-page $E_2^{p,q}$ is isomorphic to $H^p(B;\mathcal{H}^q(X;\mathbb{Z}))$, where $\mathcal{H}^q(X;\mathbb{Z})$ denotes the local system.
Since $H^1(X;\mathbb{Z}) = 0$, we have $E_2^{p, 1} = 0$ for any $p$.
Therefore we have
\[
  E_3^{3, 0} = E_2^{3, 0} = H^3(B;\mathbb{Z})
\]
and
\[
  E_3^{0, 2} = E_2^{0, 2} = H^2(X;\mathbb{Z})^{\pi_1(B)}.
\]
Here $H^2(X;\mathbb{Z})^{\pi_1(B)}$ denotes the invatiant part of the monodromy action of $\pi_1(B)$ induced from the bundle $E \to B$.
The transgression map $d_3^{0, 2}\colon E_3^{0, 2} \to E_3^{3, 0}$ defines a map
\[
  d_3^{0, 2} \colon H^2(X;\mathbb{Z})^{\pi_1(B)} = E_3^{0, 2} \to
  E_3^{3, 0} = H^3(X;\mathbb{Z}).
\]
Since the structure group of $E \to B$ is reduced to $\Symp(X,\omega)$, the monodromy action of $\pi_1(B)$ preserves the symplectic form $\omega$ on the fibers.
Together with Remark \ref{rem:inj_of_ZtoR}, the action preserves the class $[\omega]_{\mathbb{Z}} \in H^2(X;\ZZ)$.
Thus, the class $[\omega]_{\mathbb{Z}}$ is in $H^2(X;\mathbb{Z})^{\pi_1(B)}$.
By the naturality of the Serre spectral sequence, 
the cohomology class $d_3^{0,2}[\omega]_{\mathbb{Z}}$ gives rise to a characteristic class of symplectic fibrations.

\begin{proposition}[{\cite[Theorem 4.1]{carey_crowley_murray98}}]\label{prop:dd_def_ss}
  The characteristic class $-d_3^{0,2}[\omega]_{\mathbb{Z}}$ is equal to the Dixmier-Douady class $D(E)$.
\end{proposition}

Let $\Symp(X,\omega) \to E\Symp(X,\omega) \to B\Symp(X,\omega)$ be the universal $\Symp(X,\omega)$-bundle and
\[
  E = E\Symp(X, \omega) \times_{\Symp(X, \omega)} X
\]
the Borel construction.
Note that $E$ is a total space of a symplectic fibration;
\begin{align}\label{universal_symp_fibration}
  X \to E \to B\Symp(X, \omega).
\end{align}

\begin{definition}\label{def:DD_class}
  The Dixmier-Douady class $D(E) \in H^3(B\Symp(X,\omega);\ZZ)$ of fibration (\ref{universal_symp_fibration}) is called the \textit{universal Dixmier-Douady class of symplectic fibrations} and denoted by $D$.
  The class $D^{\delta} = \iota^* D \in H_{\grp}^3(\Symp(X,\omega);\ZZ)$ is called the \textit{universal Dixmier-Douady class of flat symplectic fibrations}, where $\iota^* \colon H^3(B\Symp(X,\omega);\ZZ) \to H_{\grp}^3(\Symp(X,\omega);\ZZ)$ is the canonical map (see Subsection \ref{subsec:group_coh}).
\end{definition}

\begin{lemma}
  Let $(X, \omega)$ be an integral symplectic manifold and
  \[
    BS^1 \to BQ \to B\Symp(X,\omega)
  \]
  the fibration which corresponds to the central $S^1$-extension (\ref{prequantum_extension}).
  Then, there exists a commutative diagram of fibrations
  \begin{align}\label{diagram:MtoBS^1}
    \xymatrix{
    X \ar[r] \ar[d]^{f} &E \ar[r] \ar[d]^{\phi} &B\Symp(X,\omega) \ar@{=}[d]\\
    BS^1 \ar[r] &BQ \ar[r] &B\Symp(X, \omega),
    }
  \end{align}
  where the map $f\colon X \to BS^1$ is a classifying map of the prequantization $S^1$-bundle $p\colon P \to X$.
\end{lemma}

\begin{proof}
  Since the quantomorphism group $Q$ acts on $E\Symp(X, \omega)$ through the $\Symp(X,\omega)$-action, it also acts on $E\Symp(X,\omega) \times P$ diagonally.
  Let $E\Symp(X,\omega) \times_{Q} P$ denote the quotient.
  Note that the projection $E\Symp(X,\omega) \times P \to E\Symp(X,\omega) \times X$ induces a homeomorphism
  \[
    E\Symp(X, \omega) \times_Q P \to E = E\Symp(X,\omega) \times_{\Symp(X,\omega)} X.
  \]
  Therefore, in this proof, we abuse the symbol $E$ to denote the space $E\Symp(X, \omega) \times_Q P$.
  Since $E\Symp(X,\omega) \times P \to E$ is a principal $Q$-bundle, there exists a commutative diagram
  \[
  \xymatrix{
  E\Symp(X,\omega) \times P \ar[r]^-{\Psi} \ar[d] &EQ \ar[d]\\
  E \ar[r]^-{\psi} & BQ,
  }
  \]
  where $EQ \to BQ$ is a universal $Q$-bundle and $\psi$ is a classifying map.

  Let $E\Symp(X,\omega) \times_{Q} EQ$ denotes the quotient of $E\Symp(X,\omega) \times EQ$ by the diagonal $Q$-action.
  Since the space $E\Symp(X,\omega) \times EQ$ is contractible, the principal $Q$-bundle
  \[
    E\Symp(X,\omega) \times EQ \to E\Symp(X,\omega) \times_{Q} EQ
  \]
  gives another model of the universal $Q$-bundles, and therefore the space $E\Symp(X,\omega) \times_{Q} EQ$ is one of the model of $BQ$.

  Let us consider a map
  \[
    \Phi \colon E\Symp(X,\omega) \times P \to E\Symp(X,\omega) \times EQ \ ; \  (a, p) \mapsto (a, \Psi(a,p)).
  \]
  Note that the map $\Phi$ is a bundle map of the $Q$-bundles, that is, $\Phi$ preserves the fivers and is equivariant with the $Q$-actions.
  Hence it induces a map
  \[
    \phi \colon E \to E\Symp(X,\omega) \times_{Q} EQ = BQ.
  \]
  Since the fibers of $E \to B\Symp(X,\omega)$ and $BQ \to B\Symp(X, \omega)$ are $P/S^1 = X$ and $EQ/S^1 = BS^1$, respectively.
  Therefore the restriction of the map $\phi$ to the fiber induces a map $f \colon X = P/S^1 \to EQ/S^1 = BS^1$.
  Since $f$ is covered by a bundle map $F \colon P \to EQ$ defined by $F(p) = \Psi(e, p)$, where $e \in E\Symp(X, \omega)$ is a base point, the map $f$ is a classifying map.
\end{proof}

\begin{proof}[Proof of Theorem \ref{main_theorem}]
  Let $E_{r}^{p, q}$ and $E_{r}^{'p,q}$ be the Serre spectral sequences of the fibrations $X \to E \to B\Symp(X,\omega)$ and $BS^1 \to BQ \to B\Symp(X,\omega)$ in commutative diagram (\ref{diagram:MtoBS^1}), respectively.
  Note that the space $BS^1$ has a topological group structure and the fibration $BS^1 \to BQ \to B\Symp(X,\omega)$ has a principal $BS^1$-bundle structure (see \cite[Proposition 4.1]{carey_crowley_murray98} for example).
  Since the structure group $BS^1$ is connected, the local system $\mathcal{H}^*(BS^1;\mathbb{Z})$ is trivial.
  Hence we have $E_2^{'0,2} = H^2(BS^1;\mathbb{Z})$.
  Since $H^1(BS^1;\mathbb{Z}) = 0$, we have $E_3^{'0,2} = E_2^{'0,2} = H^2(BS^1;\mathbb{Z})$ and $E_3^{'3,0} = E_2^{'3,0} = H^3(B\Symp(X,\omega);\mathbb{Z})$.
  By the naturality of the Serre spectral sequence, we obtain a commutative diagram
  \begin{align}\label{coh_comm_diag}
    \xymatrix{
    H^2(BS^1;\mathbb{Z}) \ar[r]^-{d_3^{'0,2}} \ar[d]^{f^*}& H^3(B\Symp(X,\omega);\mathbb{Z}) \ar@{=}[d]\\
    H^2(X; \mathbb{Z}) \ar[r]^-{d_3^{0,2}} &H^3(B\Symp(X,\omega);\mathbb{Z}).
    }
  \end{align}
  Since $-[\omega]_{\mathbb{Z}}$ is equal to the first Chern class of the prequantization bundle $P \to X$, we have $-[\omega]_{\mathbb{Z}} = f^*(c_1)$, where $c_1 \in H^2(BS^1;\mathbb{Z})$ is the universal first Chern class.
  By commutative diagram \ref{coh_comm_diag}, we have
  \[
  -d_3^{0,2}([\omega]_{\mathbb{Z}}) = d_3^{0,2} f^*(c_1) =
  d_3^{'0,2}(c_1) \in H^3(B\Symp(X,\omega);\mathbb{Z}).
  \]
  Let $E_{r}^{''p, q}$ denote the Serre spectral sequence of the fibration ${BS^1}^{\delta} \to BQ^{\delta} \to B\Symp(X,\omega)^{\delta}$ (or, equivalently, the Hochschild-Serre spectral sequence of the central $S^1$-extension $0 \to S^1 \to Q \to \Symp(X,\omega) \to 1$).
  Since $H^1({BS^1}^{\delta};\mathbb{Z}) = \Hom(S^1, \ZZ) = 0$, we have $E_3^{''0,2} = E_2^{''0,2} =  H^2({BS^1}^{\delta};\mathbb{Z})$.
  By the naturality of the Serre spectral sequence for the fibrations
  \[
  \xymatrix{
  {BS^1}^{\delta} \ar[r] \ar[d] & BQ^{\delta} \ar[r] \ar[d] & B\Symp(X,\omega)^{\delta} \ar[d]\\
  BS^1 \ar[r] & BQ \ar[r] & B\Symp(X,\omega),
  }
  \]
  we have a commutative diagram
  \[
  \xymatrix{
  H^2(BS^1;\mathbb{Z}) \ar[r]^-{d_3^{'0,2}} \ar[d]^{\iota^*}& H^3(B\Symp(X,\omega);\mathbb{Z}) \ar[d]^{\iota^*}\\
  H^2({BS^1}^{\delta}; \mathbb{Z}) \ar[r]^-{d_3^{''0,2}} &H^3(B\Symp(X,\omega)^{\delta};\mathbb{Z}).
  }
  \]
  Hence we obtain
  \[
    -\iota^* d_3^{0,2}([\omega]_{\mathbb{Z}}) = \iota^* d_3^{'0,2} (c_1) = d_{3}^{''0,2} \iota^* (c_1).
  \]
  Since the map $\iota\colon{BS^1}^{\delta} \to BS^1$ is a classifying map of the $S^1$-bundle $B\mathbb{Z}=S^1 \to B\mathbb{R^{\delta}} \to {BS^1}^{\delta}$, the class $\iota^* (c_1)$ is the first Chern class of the $S^1$-bundle.
  Hence the class $\iota^* (c_1)$ is equal to the obstruction class $\mathfrak{o}(\id, B\mathbb{R}^{\delta})$.
  Lemma \ref{lemma:obstruction-euler_class} for the central $\mathbb{Z}$-extension $0 \to \mathbb{Z} \to \mathbb{R} \to S^1 \to 0$ implies that the equality $\mathfrak{o}(\id, B\mathbb{R}^{\delta}) = e(\mathbb{R})$ holds.
  Hence we obtain
  \[
    -\iota^* d_3^{0,2}([\omega]_{\mathbb{Z}}) = d_{3}^{''0,2} \iota^* (c_1) = d_3^{''0,2} (e(\RR)).
  \]
  Together with Proposition \ref{prop:dd_def_ss} and Lemmas \ref{lemma:identity-Euler_class}, \ref{lemma:id_connecting_hom}, and \ref{lemma:commutative_diagram}, we obtain
  \begin{align*}
    D^{\delta} = -\iota^* d_3^{0,2}([\omega]_{\ZZ}) = d_{3}^{''0,2} (e(\mathbb{R}))= -d_3^{''0,2} \delta (\id_{S^1}) = \delta d_2^{''0,2}(\id_{S^1}) = -\delta (e(Q)).
  \end{align*}
  By Theorems \ref{thm:third_cocycle} and \ref{thm:cocycle_description}, we obtain
  \[
    D^{\delta} = -\delta (e(Q)) = \delta [a_{x_0, \alpha}] = -[c_{x_0, \sigma, w}].
  \]
\end{proof}

\section{Example}\label{sec:examples}

In this section, we show the following:
\begin{theorem}\label{thm:non-trivial_ch_class}
  Let $n$ be a positive integer.
  For the complex projective space $\mathbb{C}P^n$ with the Fubini-Study form $\omega_{FS}$, the cohomology class of $c_{x_0, \sigma, w}$ is non-zero.
\end{theorem}

Let us consider the central $S^1$-extension of the projective unitary group
\begin{equation}\label{proj_unitary_extension}
  0 \to S^1 \to U(n) \to PU(n) \to 1,
\end{equation}
where we regard $S^1$ as the unitary group $U(1)$.

\begin{lemma}\label{lem:PU(n)_bdd_coh}
  Let $e(PU(n)) \in H_{\grp}^2(PU(n);S^1)$ be the cohomology class corresponding to the central extension (\ref{proj_unitary_extension}).
  Then the class $\delta e(PU(n)) \in H_{\grp}^3(PU(n);\mathbb{Z})$ is non-zero.
  Here $\delta \colon H_{\grp}^2(PU(n);S^1) \to H_{\grp}^3(PU(n);\ZZ)$ is the connecting homomorphism.
\end{lemma}

\begin{proof}
  By Lemmas \ref{lemma:identity-Euler_class}, \ref{lemma:obstruction-euler_class}, \ref{lemma:id_connecting_hom}, and \ref{lemma:commutative_diagram}, we have
  \[
    \delta e(PU(n)) = d_3^{0,2}(\mathfrak{o}(\id, B\RR^{\delta})),
  \]
  where $d_3^{0,2}$ is the transgression map of the Serre spectral sequence of the fibration
  \[
    B{S^1}^{\delta} \to BU(n)^{\delta} \to BPU(n)^{\delta}.
  \]
  Let us consider a commutative diagram of fibrations
  \[
  \xymatrix{
  PU(n) \ar[r] \ar[d]^-{f} & EPU(n) \ar[r] \ar[d] & BPU(n) \ar@{=}[d] \\
  BS^1 \ar[r] & BU(n) \ar[r] & BPU(n) \\
  B{S^1}^{\delta} \ar[u]_-{\iota} \ar[r] & BU(n)^{\delta} \ar[u] \ar[r] & BPU(n)^{\delta} \ar[u]_-{\iota},
  }
  \]
  where $f$ is a classifying map of the $S^1$-bundle (\ref{proj_unitary_extension}).
  Let $c_1 \in H^2(BS^1;\ZZ)$ be the universal first Chern class.
  By the equality $\iota^* c_1 = \mathfrak{o}(\id,B{S^1}^{\delta})$ and the naturality, we obtain
  \[
    \delta e(PU(n)) = d_3^{0,2}(\mathfrak{o}(\id, B\RR^{\delta})) = d_3^{0,2}(\iota^* c_1) = \iota^*(d_3^{'0,2} c_1),
  \]
  where $d_3^{'0,2}$ is the transgression map of the Serre spectral sequence of $BS^1 \to BU(n) \to BPU(n)$.

  The $S^1$-bundle (\ref{proj_unitary_extension}) is non-trivial since the fundamental groups of $U(n)$ and $S^1 \times PU(n)$ are different.
  Hence the first Chern class $f^* c_1 \in H^2(PU(n);\mathbb{Z})$ of the bundle (\ref{proj_unitary_extension}) is non-zero.
  Let $E_r^{''p,q}$ be the Serre spectral sequence of the universal bundle
  \[
    PU(n) \to EPU(n) \to BPU(n).
  \]
  By the naturality between $E_r^{'p,q}$ and $E_r^{''p,q}$, we have
  \[
    \delta e(PU(n)) = \iota^*(d_3^{'0,2} c_1) = \iota^*(d_3^{''0,2} (f^* c_1)).
  \]
  Since the total space $EPU(n)$ is contractible, the transgression map $d_3^{''0,2} \colon E_3^{0,2} \to E_3^{3,0}$ is injective.
  Moreover, by Theorem \ref{theorem:milnor_made_discrete}, the map $\iota^* \colon H^3(BPU(n);\mathbb{Z}) \to H^3(BPU(n)^{\delta};\mathbb{Z})$
  is also injective.
  Hence the class $\delta e(PU(n)) = \iota^*(d_3^{''0,2} (f^* c_1)) \in H_{\grp}^3(PU(n);\mathbb{Z})$ is non-zero.
\end{proof}

Let $(X, \omega)$ be the complex projective space $\mathbb{C}P^n$
with the Fubini-Study form
$\omega_{\FS}$.
For this symplectic manifold $(\mathbb{C}P^n, \omega_{FS})$,
its prequantization bundle is the Hopf fibration
\[
  S^1 \to S^{2n + 1} \xrightarrow{p} \mathbb{C}P^n
\]
with the connection form
$\theta = \overline{z_0}dz_0 + \dots + \overline{z_n}dz_n$, where
we consider the sphere $S^{2n+1}$ as the subspace
in $\mathbb{C}^{n+1}$ with coordinate system
$(z_0, \dots, z_n)$.

\begin{proof}[Proof of Lemma \ref{thm:non-trivial_ch_class}]
  By Theorem \ref{main_theorem}, it suffices to show that the cohomology class $\delta e(Q) \in H_{\grp}^3(\Symp(\mathbb{C}P^n, \omega_{FS}); \mathbb{Z})$ is non-zero.

  Since the $U(n+1)$-action on $S^{2n+1}$ preserves the connection form $\theta$, the unitary group is contained in $Q$.
  Since the inclusion is $S^1$-equivariant, we have a commutative diagram
  \[
  \xymatrix{
  1 \ar[r] & S^1 \ar[r] \ar@{=}[d] & U(n+1) \ar[r] \ar[d] & PU(n+1) \ar[d]^{f} \ar[r] & 1 \\
  1 \ar[r] & S^1 \ar[r] & Q \ar[r] & \Symp(\mathbb{C}P^n,\omega_{FS}) \ar[r]& 1.
  }
  \]
  Hence we have $e(PU(n+1)) = f^*e(Q) \in H_{\grp}^2(PU(n+1);S^1)$.
  Let us consider a commutative diagram
  \[
  \xymatrix{
  H_{\grp}^2(\Symp(\mathbb{C}P^n,\omega_{FS});S^1) \ar[r]^-{f^*} \ar[d]^{\delta}
  & H_{\grp}^2(PU(n+1);S^1) \ar[d]^{\delta} \\
  H_{\grp}^3(\Symp(\mathbb{C}P^n,\omega_{FS});\mathbb{Z}) \ar[r]^-{f^*}
  & H_{\grp}^3(PU(n+1);\mathbb{Z}),
  }
  \]
  where each $\delta$ denotes the connecting homomorphism.
  Then we obtain
  \[
    f^* \delta e(Q) = \delta f^* e(Q) = \delta e(PU(n+1)).
  \]
  By Lemma \ref{lem:PU(n)_bdd_coh}, the last term $\delta e(PU(n+1))$ is non-zero.
  Therefore, the class $\delta e(Q) \in H_{\grp}^3(\Symp(\mathbb{C}P^n,\omega_{FS});\mathbb{Z})$ is non-zero.
\end{proof}

\section*{Acknowledgements}
The author would like to thank Professor Egor Shelukhin for telling him about the Dixmier-Douady class of Hamiltonian fibrations.
The author is supported by JSPS KAKENHI Grant Number JP21J11199.

\bibliographystyle{amsalpha}
\bibliography{ch_class_Ham_fibration.bib}
\end{document}